\documentclass[12pt]{amsart}
\usepackage{amscd}
\usepackage{graphicx}
\def\frk{\frak}               

\def\Phi{{\frk n}}
\def\Phi{{\frk N}}
\def\opn#1#2{\def#1{\operatorname{#2}}} 
\opn\chara{char} \opn\length{\ell} \opn\pd{pd} \opn\rk{rk}
\opn\projdim{proj\,dim} \opn\injdim{inj\,dim} \opn\rank{rank}
\opn\depth{depth} \opn\grade{grade} \opn\height{height}
\opn\embdim{emb\,dim} \opn\codim{codim}

\opn\Tr{Tr} \opn\bigrank{big\,rank}
\opn\superheight{superheight}\opn\lcm{lcm}
\opn\trdeg{tr\,deg}
\opn\reg{reg} \opn\lreg{lreg} \opn\ini{in} \opn\lpd{lpd}
\opn\size{size}
\opn\div{div} \opn\Div{Div} \opn\cl{cl} \opn\Cl{Cl}
\opn\Spec{Spec} \opn\Supp{Supp} \opn\supp{supp} \opn\Sing{Sing}
\opn\Ass{Ass} \opn\Min{Min}
\opn\Ann{Ann} \opn\Rad{Rad} \opn\Soc{Soc}
\opn\Im{Im} \opn\Ker{Ker} \opn\Coker{Coker} \opn\Am{Am}
\opn\Hom{Hom} \opn\Tor{Tor} \opn\Ext{Ext} \opn\End{End}
\opn\Aut{Aut} \opn\id{id}

\opn\nat{nat}
\opn\pff{pf}
\opn\Pf{Pf} \opn\GL{GL} \opn\SL{SL} \opn\mod{mod} \opn\ord{ord}
\opn\Gin{Gin} \opn\Hilb{Hilb}
\opn\aff{aff} \opn\con{conv} \opn\relint{relint} \opn\st{st}
\opn\lk{lk} \opn\cn{cn} \opn\core{core} \opn\vol{vol}
\opn\link{link} \opn\star{star}
\opn\gr{gr}

\def\pot#1#2{#1[\kern-0.28ex[#2]\kern-0.28ex]}

\opn\dirlim{\underrightarrow{\lim}}
\opn\inivlim{\underleftarrow{\lim}}
\def\Implies{\ifmmode\Longrightarrow \else
        \unskip${}\Longrightarrow{}$\ignorespaces\fi}
\def\implies{\ifmmode\Rightarrow \else
        \unskip${}\Rightarrow{}$\ignorespaces\fi}
\def\iff{\ifmmode\Longleftrightarrow \else
        \unskip${}\Longleftrightarrow{}$\ignorespaces\fi}

\let\:=\colon
\newtheorem{Theorem}{Theorem}[section]
\newtheorem{Lemma}[Theorem]{Lemma}

\newtheorem{Remark}[Theorem]{Remark}

\newtheorem{Example}[Theorem]{Example}

\newtheorem{Definition}[Theorem]{Definition}

\let\epsilon\varepsilon
\let\phi=\varphi
\let\kappa=\varkappa
\def\qed{\ifhmode\textqed\fi
      \ifmmode\ifinner\quad\qedsymbol\else\dispqed\fi\fi}
\def\textqed{\unskip\nobreak\penalty50
       \hskip2em\hbox{}\nobreak\hfil\qedsymbol
       \parfillskip=0pt \finalhyphendemerits=0}
\def\dispqed{\rlap{\qquad\qedsymbol}}

\opn\dis{dis}
\def\pnt{{\raise0.5mm\hbox{\large\bf.}}}

\opn\Lex{Lex}

\begin{document}

%


\title{On the characterization of $f$-ideals}
\maketitle
\author\begin{center}{I. Anwar$^{a}$, H. Mahmood$^{b}$,  M. A. Binyamin$^{c}$, M. K. Zafar $^{d}$}\end{center}
\begin{center}\footnotesize{\textsl{ {
$^{a}$ COMSATS Institute of Information Technology Lahore, Pakistan.\\ E. mail: imrananwar@ciitlahore.edu.pk\\
$^{b}$ Government College University Lahore, Pakistan. E. mail:  hasanmahmood@gcu.edu.pk\\
$^{c}$ National Textile University, Faislabad, Pakistan. E. mail:  ahsanbanyamin@gmail.com\\
$^{d}$ Air University Islamabad, Pakistan. E. mail:
khurram.zafar@au.edu.pk}}}
\end{center}

\begin{abstract}
In this paper, we give the complete characterization of unmixed
$f$-ideals of degree $d\geq 2$ generalizing the results given in
\cite{im}.

 \vskip 0.4 true cm
 \noindent
\end{abstract}

\section{Introduction}
Let $S=k[x_1, . . . ,x_n]$ be a polynomial ring over an infinite
field $k$.  One associates simplicial complexes to monomial ideals
in $S$.
$$\Delta \leftrightarrow I_{\mathcal{N}}(\Delta)$$
where $I_{\mathcal{N}}(\Delta)$ is known as the Stanley-Reisner
ideal or non-face ideal of $\Delta$. This one-to-one correspondence
laid the foundation of the algebraic study of simplicial
complexes, for instance see \cite{bh}, \cite{Mi} and \cite{v}.\\
In \cite{F1}, Faridi introduced another correspondence:
$$\Delta \leftrightarrow I_{\mathcal{F}}(\Delta)$$
where $I_{\mathcal{F}}(\Delta)$ is the facet ideal of a given
simplicial complex $\Delta$. Faridi in \cite{F1} and \cite{F2}
discussed algebraic properties of this ideal via combinatorial
properties of the \linebreak simplicial complex. She also discussed
its connections with the theory of  Stanley-Reisner rings.
Therefore, it is
interesting to explore some new connections between both the theories.\\
In \cite{im}, the authors investigated the following correspondence:
$$\delta_{\mathcal{F}}(I) \leftrightarrow I \leftrightarrow \delta_{\mathcal{N}}(I),$$
where $\delta_{\mathcal{F}}(I)$ and $\delta_{\mathcal{N}}(I)$ are
facet and non-face simplicial complexes associated to the
square-free monomial ideal $I$ respectively. They introduced the
concept of $f$-ideals and they gave the characterization of all
$f$-ideals of degree $2$. A square-free monomial ideal $I$ in
$S=k[x_1,x_2,\ldots,x_n]$ is said to be an $f$-ideal if and only if
both $\delta_{\mathcal{F}}(I)$ and $\delta_{\mathcal{N}}(I)$ have
the same $f$-vector. The concept of $f$-ideals is important in the
sense that it discovers new connections between both the theories
({\em facet ideal} and {\em Stanley-Reisner ideal}). In this paper,
we characterize all the unmixed $f$-ideals in the polynomial ring
$S=k[x_1,x_2,\ldots,x_n]$( see Theorem \ref{main}). Moreover, we
give a relation between the $f$-vectors of $\delta_{\mathcal{F}}(I)$
and $\delta_{\mathcal{N}}(I)$ (see Lemma \ref{faces}). We should
mention that our presentation was improved by the kind suggestions
of the Referee and Prof. Dorin Popescu.

\section{Basic Setup}
Here, we recall some elementary definitions and notions, which we
will use throughout in this paper.
\begin{Definition}
\em{A  simplicial complex $\Delta$ over a set of vertices
$V=\{x_1,x_2,\ldots,x_n\}$ is a collection of subsets of $V$, with
the property that $\{x_i\}\in \Delta$ for all $i$, and if $F\in
\Delta$ then all subsets of $F$ are also in $\Delta$(including the
empty set). An element of $\Delta$ is called a {\em face} of
$\Delta$, and the {\em dimension of a face} $F$ of $\Delta$ is
defined as $|F|-1$, where $|F|$ is the number of vertices of $F$.
The faces of dimension $0$ and $1$ are called {\em vertices and
edges}, respectively, and $\dim \emptyset=-1$. The maximal faces of
$\Delta$ under inclusion are called {\em facets}. }
\end{Definition}

\begin{Definition}\label{fvec}{\em
For a simplicial complex $\Delta$ having dimension $d$, its {\em
$f$-vector} is a $(d+1)$-tuple, defined as:
$$f(\Delta)=(f_0,f_1,\ldots,f_d),$$
where $f_i$ denotes the number of {\em $i$-dimensional faces} of
$\Delta.$ }\end{Definition}

The following definitions serve as a bridge between the
combinatorial and algebraic properties of the simplicial complexes
over the finite set of vertices $[n]$.

\begin{Definition}
{\em Let $\Delta$ be a simplicial complex over $n$ vertices
$\{v_1,\ldots , v_n\}$. Let $k$ be a field, $x_1, . . . , x_n$ be
indeterminates, and $S$ be the polynomial ring
$k[x_1,\ldots, x_n]$. \\
(a) Let $I_{\mathcal{F}}(\Delta)$  be the ideal of $S$ generated by
square-free monomials $x_{i1} \ldots x_{is}$, where $\{v_{i1} ,
\ldots, v_{is}\}$ is a facet of $\Delta$. We call
$I_{\mathcal{F}}(\Delta)$ the {\it facet ideal} of $\Delta$.\\
(b) Let $I_{\mathcal{N}}(\Delta)$  be the ideal of $S$ generated by
square-free monomials $x_{i1} \ldots x_{is}$ , where $\{v_{i1} ,
\ldots , v_{is}\}$ is not a face of $\Delta$. We call
$I_{\mathcal{N}}(\Delta)$ the {\em non-face ideal} or the {\it
Stanley- Reisner ideal} of $\Delta$.}
\end{Definition}

\begin{Definition}\label{fc}
{\em Let $I = (M_1,\ldots,M_q)$ be a square-free monomial ideal in a
polynomial ring $k[x_1, . . . , x_n]$. Furthermore, assume that
$\{M_1,\ldots ,M_q\}$ is a minimal generating set of $I$.\\
(a) Let $\delta_{\mathcal{F}}(I)$  be the simplicial complex over a
set of vertices $v_1,\ldots , v_n$ with facets $F_1,\ldots , F_q$,
where for each $i$, $F_i = \{v_j \, \, |\, \,  x_j |M_i, 1 \leq
j\leq n\}$. We call
$\delta_{\mathcal{F}}(I)$ the {\em facet complex of $I$.}\\
(b) Let $\delta_{\mathcal{N}}(I)$  be the simplicial complex over a
set of vertices $v_1,\ldots , v_n$, where $\{v_{i1},\ldots ,
v_{is}\}$ is a face of $\delta_{\mathcal{N}}(I)$ if and only if
$x_{i1} . . . x_{is} \not \in I$. We call $\delta_{\mathcal{N}}(I)$
the {\em non-face complex} or {\em the Stanley-Reisner complex} of
$I$. }
\end{Definition}
\begin{Definition} {\em Let $S=k[x_1,\ldots,x_n]$ be a polynomial ring,
the {\em support} of a\linebreak monomial $x^a =x_1^{a_1}\ldots
x_n^{a_n}$ in
$S$ is given by $\supp(x^a)=\{ x_i|a_i>0\}$.\\
Similarly, let $I=(g_1,\ldots,g_m)\subset S$ be a square-free
monomial ideal then $$\supp(I)=\bigcup_{i=1}^{m}\supp(g_i)$$ }
\end{Definition}
\begin{Remark}{\em
From \cite{im}, we know that for any square-free monomial ideal
$I\subset S$ the simplicial complex $\delta_{\mathcal{F}}(I)$ will
have the vertex set $[s]$, where $s=|\supp(I)|$. But
$\delta_{\mathcal{N}}(I)$ will be a simplicial complex on $[n]$. So
both $\delta_{\mathcal{F}}(I)$ and $\delta_{\mathcal{N}}(I)$ will
have the same vertex set if and only if
$\supp(I)=\{x_1,\ldots,x_n\}$.}
\end{Remark}
By considering the standard grading on the polynomial ring
$S=k[x_1,\ldots,x_n]$, we give the following definition.
 \begin{Definition}{\em
Let $I\subset S$ be a square-free monomial ideal with the minimal
monomial generators $g_1,\ldots,g_m$, then the $deg(I)$ is defined
as:
$$deg(I)=\sup\{deg(g_i)|\, \, i\in\{1,\ldots,m\}\}$$}
\end{Definition}

For more details, see \cite{im}, \cite{F1} and \cite{HP}.

\section{characterization of $f$-ideals.}
Here we recall some definitions from \cite{im};
\begin{Definition}{\em
Let  $I\subset S$ be a square-free monomial ideal with a
minimal\linebreak generating system $\{g_1,\ldots ,g_m\}$.  We say
that $I$ is a {\em pure square-free monomial ideal of degree $d$} if
and only if $\supp(I)=\{x_1,\ldots,x_n\}$ and $deg(g_i)=d>0$ for all
$1\leq i\leq m$. }
\end{Definition}

\begin{Definition}{\em Let $I$ be a square free monomial ideal in $S=k[x_1,x_2,....,x_n]$.
 We say that $I$ is an $f$-ideal if and only if $f(\delta_\mathcal{F}(I))=f(\delta_\mathcal{N}(I))$.}
  \end{Definition}
There is a natural question to ask: {\em  characterize all the
$f$-ideals in $S$}. In \cite{im}, authors precisely gave the
characterization of $f$-ideals of degree $2$. Next we extend this
result for $f$-ideals of degree $\geq 3$.

\begin{Theorem}\label{nec}{{\bf (necessary conditions) }}{\em Let $I=(g_1,g_2,\ldots,g_s)$ be a pure square-free monomial
 ideal of degree $d$ with $d\geq 2$  in $S=k[x_1,x_2,\ldots,x_n]$. If $I$ is an unmixed $f$-ideal, then $I$ satisfies the following conditions;\\
$(1)$ $I$ is of height $n-d$;\\
$(2)$ ${n \choose d}\equiv 0 \ \ \ (mod\ 2)$;\\
$(3)$ $s=\mid Ass(S/I)\mid=\frac{1}{2}{n \choose d}$.
 }

\end{Theorem}

\begin{proof}

Since $I$ is an unmixed $f$-ideal and $deg(I)=d$,  $dim
\delta_{\mathcal{F}}(I)=dim\delta_{\mathcal{N}}(I)$ if and only if
$ht(I)=n-d$ (by \cite[Lemma 3.4]{im}). Now as
$f_{d-1}(\delta_{\mathcal{F}}(I))=s$ from \cite[Lemma 3.2]{im} it
follows $f_{d-1}(\delta_{\mathcal{N}}(I))={n \choose d}-s$. But $I$
is an $f$-ideal and so $s={n \choose d}-s$ which gives
$s=\frac{1}{2}{n \choose d}$. Using \cite[Lemma 3.2]{im}, we see
that $\delta_{\mathcal{N}}(I)$ is a pure simplicial complex
generated by $\frac{1}{2}{n \choose d}$ facets of dimension $d-1$.
Thus \cite[Lemma 5.3.10]{v} gives that $I$ is unmixed of height
$n-d$ with $\mid Ass(S/I)\mid=\frac{1}{2}{n \choose d}$.

 \end{proof}

 \begin{Remark}{\em The above three conditions are also sufficient conditions  for a pure square
 free unmixed
 monomial ideal of degree $2$ to be an $f$-ideal as it is proved in \cite{im}. Below We  give an example
  to show that for $d\geq3$ above mentioned conditions  are not sufficient.}
 \end{Remark}

\begin{Example}{\em Consider the square-free monomial ideal $I$ in the polynomial ring $S=k[x_1,x_2,x_3,x_4,x_5]$ given by
 $$ I=(x_1x_2x_4, x_1x_2x_5\ ,x_3x_4x_5,\ x_1x_4x_5,\ x_2x_3x_5)$$
 \ \ \ \ \ \ \ \ \ \ \ \ \ \ \ \ \ \ \ \ \ \ \ \ \ \ \ $=(x_1,x_3)\cap (x_1,x_5)\cap (x_2,x_4)\cap (x_2,x_5)\cap (x_4,x_5)$\\
 Note that\\
 $(1)$ $I$ is unmixed of height $5-3=2$\\
 $(2)$ ${5 \choose 3}=10\equiv 0 \ \ \ (mod\ 2)$\\
 $(3)$ $s=\mid Ass(S/I)\mid=\frac{1}{2}{5 \choose 3}=5$\\
 Its facet complex $\delta_{\mathcal{F}}(I)$ and Stanley-Reisner complex $\delta_{\mathcal{N}}(I)$ are as follows\\
 $$\delta_{\mathcal{F}}(I)=<\{1,2,4\},\ \{1,2,5\},\ \{1,4,5\},\ \{2,3,5\},\ \{3,4,5\}>$$
 $$\delta_{\mathcal{N}}(I)=<\{2,4,5\},\ \{2,3,4\},\ \{1,3,5\},\ \{1,3,4\},\ \{1,2,3\}>$$
 But $f(\delta_{\mathcal{F}}(I))=(5,\ 9,\ 10)\neq f(\delta_{\mathcal{N}}(I))=(5,\ 10,\ 10)$. Thus $I$ is not $f$-ideal.}
 \end{Example}
In order to characterize the $f$-ideals, we need the following
lemmas.
\begin{Lemma}\label{faces} {\em Let $I=(g_1, g_2, \ldots, g_s)$ be pure square-free monomial ideal of degree $d$ in
$S=k[x_1,x_2,...,x_n]$. If $F$ is a face in
$\delta_{\mathcal{F}}(I)$ of dimension less than $d-1$, then $F$
belongs to $\delta_{\mathcal{N}}(I)$ as well. In particular;
$$f_i(\delta_{\mathcal{F}}(I))\leq f_i(\delta_{\mathcal{N}}(I))\hbox{\, \, \, for all\, \, \,    } \, i<d-1.$$}\end{Lemma}
\begin{proof}
Let us take $F=\{j_1,j_2,\ldots, j_r\}$ with $r<d$ is a face in
$\delta_{\mathcal{F}}(I)$, then by definition there exists a
monomial $m=x_{j_1}\ldots x_{j_r}$ such that $m$ divides $g_i$ for
some $i(1\leq i\leq s)$. Suppose on the contrary that $F\not\in
\delta_{\mathcal{N}}(I)$, then by definition, we have
$m=x_{j_1}\ldots x_{j_r}\in I$ which is a contradiction as $m$ is a
monomial of degree strictly less than $d$.\end{proof}
\begin{Lemma}\label{d-1}{\em
Let $I=(g_1, g_2,\ldots,g_s)\subset S=k[x_1,x_2,\ldots,x_n]$ be an
$f$-ideal of degree $d$, then $$f_{d-2}(\delta_{\mathcal{F}}(I))={n
\choose d-1}.$$ }\end{Lemma}
\begin{proof}
Suppose on contrary that there exists a face
$F=\{j_1,j_2,\ldots,j_{d-1}\}\not\in\delta_{\mathcal{F}}(I)$. But
$F$ will be contained in some face $G$ with $dim(G)=d-1$. It is
clear from\linebreak Theorem \ref{nec} that $G\in
\delta_{\mathcal{N}}(I)$. Consequently, we have $F\in
\delta_{\mathcal{N}}(I)$. From  Lemma \ref{faces}, it is clear that
$\delta_{\mathcal{N}}(I)$ contains all the faces of
$\delta_{\mathcal{F}}(I)$ with dimension strictly less than $d-1$.
Therefore, we have
$f_{d-2}(\delta_{\mathcal{N}}(I))>f_{d-2}(\delta_{\mathcal{F}}(I))$
which is a contradiction.

\end{proof}
\begin{Theorem}\label{main}{{\bf (characterization) }}{\em Let $I\subset S$ be a
pure unmixed square-free monomial ideal of degree $d$ with the
minimal set of generators $\{g_1,g_2,\ldots,g_s\}$. Then $I$ will be
an unmixed $f$-ideal if and only if the following
conditions are satisfied.\\
$(1)$ $I$ is of height  $n-d$, \\
$(2)$ ${n \choose d}\equiv 0\, \, \, (\mod 2) \hbox{\, and
\,}|Ass(S/I)|=s=\frac{1}{2}{n \choose d}.$\\
$(3)$ $f_{d-2}(\delta_{\mathcal{F}}(I))={n \choose d-1}$.}

\end{Theorem}

\begin{proof} If $I$ is an $f$-ideal, then apply Theorem \ref{nec} and Lemma \ref{d-1}.
Now suppose that  $I=(g_1,g_2,\ldots, g_s)$ is a pure unmixed
square-free monomial ideal of degree $d$ satisfying the conditions
$(1)$, $(2)$, and $(3)$. Then by Lemma \ref{faces} and \ref{d-1} it
follows that
$f_{i}(\delta_\mathcal{F}(I))=f_{i}(\delta_\mathcal{N}(I))={n
\choose i+1}$ for each $i\in\{1,2,....,d-2\}$ and
$f_{d-1}(\delta_\mathcal{F}(I))=f_{d-1}(\delta_\mathcal{N}(I))=s$
using \cite[Lemma 3.2]{im}. \end{proof}

Here we give an example of $f$-ideal in degree 3.
\begin{Example}{\bf{($f$-ideal of degree 3)}}\ {\em Let us consider the ideal $I$ in the polynomial ring $S=k[x_1,x_2,x_3,x_4,x_5,x_6]$ given by
\\
\\
$I=(x_1x_2x_3, x_1x_2x_5\ ,x_1x_3x_4,\ x_1x_4x_5,\ x_1x_5x_6,\ x_2x_3x_4,\ x_2x_3x_6,\ x_2x_4x_6,\ x_3x_4x_5,\ x_3x_5x_6).$\\
\\
Clearly, $\supp(I)=\{x_1,x_2,x_3,x_4,x_5,x_6\}$ . The primary
decomposition of $I$ shows that $I$ is unmixed of height 3.

$$I=(x_2,x_4,x_5)\cap (x_2,x_3,x_5)\cap (x_1,x_2,x_5)\cap (x_3,x_4,x_5)\cap (x_1,x_3,x_4)$$
 \ \ \ \ \ \ \ \ \ \ \ \ \ \ \ \  $\cap (x_1,x_2,x_3)\cap (x_1,x_3,x_6)\cap (x_3,x_5,x_6)\cap (x_1,x_4,x_6)\cap (x_2,x_4,x_6)$

$$\delta_{\mathcal{F}}(I)=<\{1,2,3\},\ \{1,2,5\},\ \{1,3,4\},\ \{1,4,5\},\ \{1,5,6\},$$
\ \ \ \ \ \ \ \  \ \ \ \ \ \ \ \ \ \ \ \ \ \ \ \ \ \ \ \ \ \ \    $\{2,3,4\},\ \{2,3,6\},\ \{2,4,6\},\ \{3,4,5\},\ \{3,5,6\}>$

$$\delta_{\mathcal{N}}(I)=<\{1,3,6\},\ \{1,4,6\},\ \{3,4,6\},\ \{1,2,6\},\ \{2,5,6\},$$
\ \ \ \ \ \ \ \ \ \ \ \ \ \ \ \ \ \ \ \ \ \ \ \ \ \ \ \ \ \ \  $\{4,5,6\},\ \{2,4,5\},\ \{1,2,4\},\ \{2,3,5\},\ \{1,3,5\}>$
\\

Note that $f$-vectors of the facet complex and the non-face complex of $I$ are same i.e. $f(\delta_{\mathcal{F}}(I))=f(\delta_{\mathcal{N}}(I))=(6,\ 15,\ 10)$ which shows that $I$ is an $f$-ideal of degree 3. }
\end{Example}
\begin{Remark} \em{ By \cite[Proposition 5.4.4]{v}, we see that for a given square-free monomial
ideal $I$ the {\em Hilbert Series} of the quotient ring $S/I$ can be
obtained from the $f$-vector of $\delta_{\mathcal{N}}(I)$. But such
result is not known for the $\delta_{\mathcal{F}}(I)$. It is worth
noting that if $I$ is an $f$-ideal, then one may obtain as well from
the $f$-vector of $\delta_{\mathcal{F}}(I)$ by using
\cite[Proposition 5.4.4]{v}. Moreover, one may explain the
$f$-ideals for more. }
\end{Remark}

\end{document}